\documentclass[11pt,reqno]{amsart}

\usepackage[sort]{cite}
\usepackage{amsfonts} 
\usepackage{amsmath} 
\usepackage{amssymb} 
\usepackage{amsthm} 
\usepackage{latexsym} 
\usepackage{mathrsfs} 
\usepackage{mathtools} 
\mathtoolsset{showonlyrefs} 
\usepackage{slashed}  

\usepackage{verbatim}
\usepackage{cases}
\usepackage[dvips]{epsfig}
\usepackage{epsf} 
\usepackage[bookmarksnumbered,pdfpagelabels=true,plainpages=false,colorlinks=true,
            linkcolor=black,citecolor=black,urlcolor=black]{hyperref}
\usepackage{url}
\usepackage{color}
\usepackage{xcolor}
\usepackage{graphicx}
\usepackage{enumitem} 

\usepackage{pifont} 
\usepackage{upgreek} 
\usepackage{fancyhdr} 
\usepackage{calligra} 
\usepackage{marvosym} 
\usepackage[percent]{overpic} 
\usepackage{pict2e} 
\usepackage[hypcap=false]{caption} 
\usepackage{wasysym} 
\usepackage{accents}
\usepackage[papersize={8.5in,11in}, margin=1in,footskip=0.4in,vmarginratio={1:1},hmarginratio={1:1}]{geometry} 
\newcommand{\la}{\langle}
\newcommand{\ra}{\rangle}

\theoremstyle{plain}
\newtheorem{theorem}{Theorem}[section]

\newtheorem{lemma}[theorem]{Lemma}
\newtheorem{proposition}[theorem]{Proposition}
\newtheorem{corollary}[theorem]{Corollary}

\theoremstyle{definition}
\newtheorem{remark}[theorem]{Remark}

\newtheorem{definition}[theorem]{Definition}

\numberwithin{equation}{section}
\allowdisplaybreaks

\newcommand{\norm}[1]{\|#1\|}

\newcommand{\cA}{\mathcal{A}}
\newcommand{\cB}{\mathcal{B}}
\newcommand{\cC}{\mathcal{C}}
\newcommand{\cD}{\mathcal{D}}

\newcommand{\cL}{\mathcal{L}}

\newcommand{\cN}{\mathcal{N}}
\newcommand{\cP}{\mathcal{P}}
\newcommand{\cR}{\mathcal{R}}
\newcommand{\cS}{\mathcal{S}}

\newcommand{\bB}{\mathbb{B}}

\newcommand{\bE}{\mathbb{E}}
\newcommand{\bN}{\mathbb{N}}
\newcommand{\bR}{\mathbb{R}}
\newcommand{\bS}{\mathbb{S}}
\newcommand{\bT}{\mathbb{T}}
\newcommand{\bZ}{\mathbb{Z}}

\newcommand{\fF}{\mathfrak{F}}
\newcommand{\fR}{\mathfrak{R}}
\newcommand{\fS}{\mathfrak{S}}
\newcommand{\fU}{\mathfrak{U}}

\newcommand{\proj}{\mathsf{\Pi}}

\newcommand{\ii}{\mathrm{i}}
\newcommand{\tm}{\widetilde{M}}

\begin{document}
\title[Strongly hyperbolic systems]{Strongly hyperbolic quasilinear systems revisited, with applications to relativistic fluid dynamics}
\author[Disconzi, Shao]{Marcelo M. Disconzi$^{* \#}$, 
Yuanzhen Shao$^{** \dagger}$}
	
\thanks{$^{\#}$MMD gratefully acknowledges support from NSF grant DMS-2107701, from a Chancellor's Faculty Fellowship, and a Vanderbilt's Seeding Success grant.
}

\thanks{$^{\dagger}$YS gratefully acknowledges support from NSF grant DMS-2306991.}

\thanks{$^{*}$Vanderbilt University, Nashville, TN, USA.
\texttt{marcelo.disconzi@vanderbilt.edu}}

\thanks{$^{**}$The University of Alabama, Tuscaloosa, AL, USA.
\texttt{yshao8@ua.edu}}

\begin{abstract}
We revisit the theory of first-order quasilinear systems with diagonalizable principal part and only real eigenvalues, what is commonly referred to as strongly hyperbolic systems. We provide a self-contained and simple proof of local well-posedness, in the Hadamard sense, of the Cauchy problem. Our regularity assumptions are very minimal. As an application, we apply our results to systems of ideal and viscous relativistic fluids, where the theory of strongly hyperbolic equations has been systematically used to study several systems of physical interest.
\bigskip

\noindent \textbf{Keywords:} Strong hyperbolicity, first-order quasilinear systems, relativistic fluids.

\bigskip

\noindent \textbf{Mathematics Subject Classification (2020):} 
Primary: 35L60; 
Secondary: 	
	35Q35,  
35Q75. 

\end{abstract}

\maketitle

\tableofcontents

\section{Introduction}\label{S:Intro}

We will study the Cauchy problem for a first-order quasilinear hyperbolic system of the form
\begin{equation}
\label{1st_order_sys_quasilinear}
\partial_t u + \cA^i (u) \partial_i  u    =  \cR(u), \quad u(0)=u_0,  \quad i=1,2,\cdots,N,
\end{equation}
where $\cA^i : \bR\times   \bT^N \times \bR^m \to \bR^{m^2}$, $ \cR: \bR\times   \bT^N \times \bR^m \to \bR^m$, $u: \bR\times  \bT^N \to \bR^m$, and
$$
\cN(u)( t, x)= \cN( t, x, u( t, x)) ,\quad \cN\in \{\cA^i,\cR\},
$$
where $(t,x) \in \bR \times \bT^N$.
Here and throughout, we use the summation convention so that repeated indices are summed over their range. 
By hyperbolicity of  \eqref{1st_order_sys_quasilinear}, we mean that the $m\times m$ matrix
\begin{equation}
\label{symobol-hyperbolic}
\cA(t,x,\zeta,\xi)=\cA^i(t,x,\zeta)\xi_i,   \quad \xi=(\xi_1,\cdots, \xi_N)   
\end{equation}
is diagonalizable and has only real eigenvalues for all $(t, x,\zeta,\xi)\in \bR\times   \bT^N \times \overline{U} \times \bS^{N-1}$ for some non-empty open set $U\subset \bR^m$.

In this article, we restrict our discussion to $\bT^N$ for simplicity, though, with minor modifications, all results can be generalized to Cauchy problems on any compact manifold ${\sf M}$ without boundary or $\bR^N$ with compactly supported initial data. Moreover, for most systems of interest, solutions enjoy a domain-of-dependence property, so that the problem can be localized and global-in-space solutions constructed out of our results in $\bT^N$. In this case, one can also easily adapt 
our results for the case where Sobolev spaces are replaced by uniformly local Sobolev spaces.

The theory of symmetric and symmetrizable hyperbolic systems, i.e. in \eqref{1st_order_sys_quasilinear} the matrix $\cA(t,x,\zeta,\xi)$ is symmetric or there exists a symmetrizer $S(t,x,\zeta,\xi)$ such that
$$
S(t,x,\zeta,\xi)\cA(t,x,\zeta,\xi) \text{ is symmetric},
$$
serves as  a powerful tool for studying many important problems  in mathematical physics such as the Einstein field equations of general relativity  \cite{Ringstrom-Book-2009,Fischer-Marsden-1972,Hughes-Kato-Marsden-1976}, the  compressible magnetohydrodynamics equations \cite{Chandrasekhar-1961}, the Euler equation of gas dynamics \cite{Majda-Book-1984}, and the relativistic Euler equations and relativistic magnetohydrodynamic Equations \cite{Anile-Book-1990}, to cite only a few well-known examples.
As a theoretical foundation of these problems, the local well-posedness theory for  the Cauchy problem of symmetric and symmetrizable hyperbolic systems has been studied by numerous authors; it is not feasible here to review the large literature on this topic, so we restrict ourselves to point out the following classical papers and monographs \cite{Courant-Hilbert-Book-1989-2, Dunford-Schwartz-Book-1958-I, Lax-1955, Lax-Book-1973, Friedrichs-1954, Kato-1975-1, Taylor-Book-1996-1, Taylor-Book-1996-2, Taylor-Book-1996-3,Majda-Book-1984}.
In the classic theory, it is known that,  for every initial datum $u_0$  belonging to the $L^2$-based Sobolev space $H^s=H^s(\bT^N)$   with integer order $s>n/2+1$, a quasilinear symmetric or symmetrizable system admits a unique solution $u\in C([-T,T],H^s)\cap C^1([-T,T],H^{s-1})$ for some $T>0$.

We notice that some techniques in the study of  symmetric and  symmetrizable hyperbolic systems  can be modified and generalized to a large class of hyperbolic systems, in which only some   mild regularity condition on the eigenvalues of $\cA(x,\zeta,\xi)$ is assumed, see Assumption~(A4) in Section~\ref{Section:main results}.
Such a  structural assumption on \eqref{1st_order_sys_quasilinear}, to the best of our knowledge,  is in some sense minimal\footnote{See \cite{Disconzi-Luo-Mazzone-Speck-2022,Wang-2022,Yu-2022-arxiv,Zhang-Andersson-arxiv-2022} for low-regularity results for the compressible and relativistic Euler equations, where additional regularity for the vorticity needs to be assumed.}.
It is the goal  of this article   to develop the well-posedness theory for  hyperbolic systems under such a condition.  
This theory has been implicitly stated in our previous work \cite{Bemfica-Disconzi-Rodriguez-Shao-2021} (see also \cite{Bemfica-Disconzi-Graber-2021}) and was used to establish the local well-posedness of a relativistic viscous fluid model in Sobolev spaces,
and some variant of it is also implicit in classical works such as \cite{Taylor-Book-1996-1,Taylor-Book-1996-2,Taylor-Book-1996-3,Taylor-Book-1991}.
We nevertheless find it of particular interest to systematically develop the theory and conduct a complete study of the related results, providing a simple and self-contained proof of local well-posedness.

The rest of the paper is organized as follows. In Section~\ref{Section:main results}, we collect some preliminaries and state the main results of the paper.
The theoretical foundation of our proofs is a novel energy estimate for the   linear system associated to \eqref{1st_order_sys_quasilinear}, which will be derived in Section~\ref{Section:linear}.
Then the  existence and uniqueness of a solution to the quasilinear system~\eqref{1st_order_sys_quasilinear} can be obtained by modifying the approximation argument in the classic approach to symmetric hyperbolic systems.
These ideas will be explained in details in Section~\ref{Section:existence and uniqueness}.
Some supplemental results such as continuation criterion will be stated in the same section.
To illustrate these results, we will give  applications in Section~\ref{Section:applications}.

{\bf Notations:}
Before stating our main results, we set down some notations. 

In this article, $r,s,k,n,m\in \bN$ and $\bR_+=(0,\infty)$. In addition, $\la \nabla \ra := (1 -\Delta)^{\frac 12}$.

For any finite-dimensional Banach space $E$, we denote the $E$-valued $L^2$ based Sobolev space of order $r$ defined on $\bT^N$ by $H^r(\bT^N;E)$, with norm $\| \cdot \|_r$. In particular, $\|\cdot\|_0$ stands for the $L^2$-norm.
When the domain and the range of the space $H^r(\bT^N;E)$ is clear from the context, we will simply denote it by $H^r$.

Let $I=[-T,T]$ for some $T>0$ and
\begin{align}
\nonumber
\bB(I) = C(I; C^1) \cap C^1(I; C),
\end{align}
where $C^k = C^k(E)$ is the space of $k$-times continuously differentible functions on $E$ with its standard norm and we denote $C = C^0$.
For any $s\in \bN$, set
\begin{align}
\nonumber
\bE^s(I) = C(I; H^s) \cap C^1(I; H^{s-1}).
\nonumber
\end{align}
$ F_I:\bR_+\to \bR_+$  ($C_I $, resp.) denotes  a continuous increasing function (a  positive constant, resp.) depending on the structure of $\cA^i$ and $\cR$ such that 
$$
F_{I_1} \leq F_{I_2}   \, (C_{I_1} \leq C_{I_2} \text{, resp.)} \quad \text{whenever }I_1\subset I_2.
$$

We denote by $\widehat{f}$ the Fourier transform\footnote{If adapting our results for $\bR^N$, $\xi \in \bR^N$ and a normalization factor introduced, in the usual fashion.} of $f:\bT^N \rightarrow \bR^m$,
$$
\widehat{f}(\xi) = \int_{x \in \bT^N} e^{-\ii x \cdot \xi} f(x) \, dx,   \quad \xi \in \bZ^N,
$$
where $\cdot$ is the standard Euclidean inner product. 

If $M$ is a (possibly complex-valued) matrix, $M^*$ denotes its adjoint. 
We denote by $\cL(H)$ the space of continuous linear maps from a Banach space $H$ into itself with its usual norm. We denote the $L^2$-inner product by $(\cdot,\cdot)$.

\section{Preliminaries and the Main Results}\label{Section:main results}

\subsection{Preliminaries}

Throughout this article, we assume that there exists a non-empty open set $U\subset \bR^m$ such that the following conditions on $\cA^i$ and $\cR$ hold.
\begin{itemize}
\item[(A1)] $\cA^i \in C^\infty(\bR\times   \bT^N\times \overline{U}; \bR^{m^2})$ satisfies that for any multi-indices $k\in \bN$, $\alpha \in \bN^n$ and $\beta \in \bN^m$, there is a continuous increasing function $f_{k,\alpha,\beta}^i$ such that
$$
|\partial_t^k \partial_x^\alpha \partial_\zeta^\beta \cA^i(t, x,\zeta)|\leq f_{k, \alpha,\beta}^i(|\zeta|).
$$
\item[(A2)] $\cR \in C^\infty( \bR\times  \bT^N\times \overline{U}; \bR^m)$ satisfies that for any multi-indices $k\in \bN$, $\alpha\in \bN^n$ and $\beta\in \bN^m$, there is a continuous increasing function $h_{k,\alpha,\beta}$ such that
$$
| \partial_t^k \partial_x^\alpha \partial_\zeta^\beta \cR(t,  x,\zeta)|\leq h_{k,\alpha,\beta} (|\zeta|).
$$
\item[(A3)] For any $\xi=(\xi_1,\cdots,\xi_N)\in \bR^N \setminus \{0\}$ and $\zeta \in \overline{U}$, the matrix $\cA=\cA(t,  x,\zeta,\xi)=  \cA^i (t,  x,\zeta)\xi_i$ is diagonalizable and has only real eigenvalues, i.e. there exist a diagonal matrix $\cD=\cD(t,  x,\zeta,\xi)$ with real numbers on its diagonal and a real  invertible matrix $\cS=\cS(t,  x,\zeta,\xi)$ such that
$$
\cS\cA  =  \cD\cS. 
$$
Note that we can always normalize the rows of $\cS$ so that it is homogeneous of degree $0$ in $\xi$.
In addition, we set $\cS(t,  x,\zeta,0)=I_m$ for all $(t,  x,\zeta)\in \bR\times \bT^N\times \overline{U}$.
\item[(A4)]  
The distinct eigenvalues, 
$$\lambda_1<\lambda_2<\cdots<\lambda_l,$$
of $\cA$ satisfy $\lambda_i\in C^{1, 0,1,0}(\bR\times  \bT^N\times \overline{U} \times (  \bR^N \setminus \{0\}))$ and the $\lambda_i$'s do not meet each other on  $\bR\times  \bT^N\times \overline{U} \times (  \bR^N \setminus \{0\})$, $i=1,2,\cdots,l$, i.e., $\lambda_i(t,x,\zeta,\xi) \neq  \lambda_j(t,x,\zeta,\xi)$ for all $(t,x,\zeta,\xi) \in \bR\times  \bT^N\times \overline{U} \times (  \bR^N \setminus \{0\})$ if $i\neq j$, $i,j=1,\dots, l$.
\end{itemize} 

\medskip
\begin{remark}
The differentiability conditions (A1) and (A2) are standard in the study of quasilinear hyperbolic equations, cf. \cite{Ringstrom-Book-2009}.
(A3)	 is basically a restatement of the hyperbolicity of System~\eqref{1st_order_sys_quasilinear}.
\end{remark}

\begin{lemma}
\label{Lem: P}
The matrix $\cP=\cP(t,  x, \zeta,\xi):=(\cS^*\cS)(t,  x, \zeta,\xi)$ is homogeneous of degree $0$ in $\xi$ and   satisfies
$$\cP\in C^{1, 0,1,0}(\bR\times  \bT^N\times \overline{U} \times ( \bR^N \setminus \{0\}); \bR^{m^2}).$$ 
\end{lemma}
\begin{proof}
First, note that all eigenvalues $\lambda_i(t,  x,\zeta, \xi)$ are homogeneous of degree $1$ in $\xi$ for all $(t, x,\zeta)\in \bR\times  \bT^N\times \overline{U}$. 
The orthogonal projection onto the eigenspace of  $\lambda_i$ is given by
\begin{align}
\nonumber
\cP_i=\cP_i(t,  x,\zeta, \xi)= \frac{1}{2\pi \ii}\int_{\gamma_i} (z - \cA(t,  x,\zeta,\xi))^{-1}\, dz, \quad i=1,2,\cdots,l,
\end{align}
where $\gamma_i=\gamma_i(t,  x, \zeta, \xi)$ is a Jordan contour enclosing only one pole $\lambda_i$ homogeneous of degree $1$ in $\xi$. 
We can thus infer that $\cP_i$ is homogeneous of degree $0$ in $\xi$.

For every fixed $(t_0, x_0,\zeta_0,\xi_0)\in \bR\times  \bT^N \times \overline{U} \times ( \bR^N \setminus \{0\})$, by the Cauchy integral theorem, $\gamma_i$ can be uniformly chosen for all $(t,  x ,\zeta, \xi )$ in a neighborhood of $(t_0,  x_0,\zeta_0,\xi_0)$. Based on this assumption, one can easily verify that 
$$\cP_i\in C^{1,  0,1,0}( \bR\times  \bT^N \times  \overline{U} \times (\bR^N \setminus \{0\}); \bR^{m^2}).$$ 
The assertion follows immediately due to the fact that $\cP=\sum\limits_{i=1}^l \cP_i$.
\end{proof}

\begin{remark}\label{Rmk: est of A and R}
(i) Based on (A1) and (A2), following a standard argument, given any compact interval $I\subset\bR$ and $s>N/2$, it holds  for any $u\in H^s$ such   that
$$
\| \cN(u) \|_s \leq C_I + F_I(\|u\|_\infty) \|u\|_s,\quad \cN\in \{\cA^i,   \cR\};
$$
and given any $k\in \bN$,  it holds  for any $u\in C^k$ that
$$
\| \cN(u) \|_{k,\infty} \leq   F_I(\|u\|_\infty) \|u\|_{k,\infty},\quad \cN\in \{\cA^i,   \cR\}.
$$

(ii) By Lemma~\ref{Lem: P}, the   smallest  eigenvalues of $\cP( t, x,\zeta,\xi)$ depends continuously on  $(t, x,\zeta,\xi)\in  \bR\times \bT^N\times \overline{U}\times (\bR^N\setminus\{0\})$. Therefore,  there exists a continuous and increasing function  $\Lambda_0 :\bR_+ \to \bR^+  $ depending on $U$ only such that for any $\upzeta\in \bR^m$
\begin{align}
\label{norm equivalence}
\begin{split}
\Lambda_0(R) |\upzeta|^2\leq \upzeta^T \cP( t, x,\zeta,\xi) \upzeta \leq \Lambda_1  |\upzeta|^2,\quad ( t, x,\zeta,\xi)\in  \bR\times  \bT^N \times
(\overline{U} \cap B_R)\times  \bR^N,
\end{split}
\end{align}
where $B_R=\{\zeta\in \bR^m: |\zeta|\leq R\}$ and   $\Lambda_1$ is a constant.
\end{remark}

\subsection{Main Results}

\begin{theorem}
\label{Thm:main_theorem 1}
Let $r>N/2+1$. Assume that (A1)-(A4) are satisfied on some  $U\subset \bR^m$     open. Then, for any $u_0\in H^r$ with $\tm={\rm dist}(U^c, {\rm Im}(u_0))>0$, there exists some $T>0$ depending on $\norm{u_0}_r$ and $\tm$ such that \eqref{1st_order_sys_quasilinear} has a unique solution 
$ 
u\in \bB(I)
$ 
on $I=[-T,T]$. Moreover, the solution satisfies
\begin{align}
\label{Thm 1:regularity}
u\in C(I; H^r) \cap C^1(I; H^{r-1})
\end{align}
and for $t\in [-T ,T]$
\begin{align}
\label{Thm 1:energy estimate}
\|u(t)\|_r^2\leq F_I(\norm{u}_{C(I\times \bT^N)}) e^{ |t|  F_I(\norm{u}_{\bB(I)}) } \left( \norm{u_0}_r^2 + |t|    F_I(\norm{u}_{\bB(I)}) \right)  
\end{align}
and for some continuously differentiable and strictly increasing function $K_I:\bR_+\to \bR_+ $
\begin{align}
\label{Thm 1:energy estimate 2}
\|u(t)\|_r^2\leq  K_I^{-1} \left( K_I(  F_I(\norm{u}_{C(I\times \bT^N)}) \norm{u_0}_r^2) + |t|  F_I(\norm{u}_{C(I\times \bT^N)})    \right)  
\end{align}
and
\begin{align}
\label{Thm 1:constraint}
\bigcup_{t\in I} {\rm Im}(u(t)) \subset U.
\end{align}
\end{theorem}

\begin{theorem}
\label{Thm:main_theorem 3}
Let $r>N/2+1$. Assume that (A1)-(A4) are satisfied on some  $U\subset \bR^m$     open.  
The solution asserted in Theorem~\ref{Thm:main_theorem 1} depends continuously on the initial data in the following sense.
Assume that  $u_0\in H^r$ with $\tm={\rm dist}(U^c, {\rm Im}(u_0))>0$ and $u$ is the solution asserted in Theorem~\ref{Thm:main_theorem 1} with initial datum $u_0$ on the interval $I=[-T,T]$. 
If $u_{0,n}\to u_0$ in $H^r$ as $n\to \infty$, then for sufficiently large $n$, the solutions $u_n$ corresponding to the initial data $u_{0,n}$ also exist on $I$ and $u_n \to u$ in $C(I,H^r)$.
\end{theorem}

\begin{theorem}
\label{Thm:main_theorem 2}
Let $r>N/2+1$. Assume that (A1)-(A4) are satisfied on some  $U\subset \bR^m$     open. 
Given any $I=[0,T]$ for some $T>0$, if 
$u\in \bB(I)$ satisfying \eqref{Thm 1:constraint}   is a solution to \eqref{1st_order_sys_quasilinear} with $u_0\in H^r$, then $u$ satisfies \eqref{Thm 1:regularity} and \eqref{Thm 1:energy estimate}.

Let $T^*$ be the supremum times $T$ such that there is a solution  $u\in \bB(I)$ to \eqref{1st_order_sys_quasilinear} on $I=[0,T]$ satisfying \eqref{Thm 1:constraint}. Then one and only one of the following statements holds.
\begin{itemize}
\item $T^*=\infty$
\item  $\inf\limits_{0\leq t < T^*} {\rm dist}(U^c , {\rm Im}u(t))=0$
\item $\lim\limits_{T\to T^*_-}   \norm{u}_{\bB([0,T])}=+\infty$
\end{itemize}
Similar results holds on $[-T,0]$.
\end{theorem}

\begin{remark}
The techniques used to prove
Theorems~\ref{Thm:main_theorem 1} and \ref{Thm:main_theorem 2} are not completely new. They are   essentially an improvement of the  standard approach to symmetrizable systems, cf. \cite[Section~16.2]{Taylor-Book-1996-1}.
However, since these results have never been systematically studied or explicitly stated in previous literature, 
we find it of  independent interest to give an self-contained proof for the assertions in Theorems~\ref{Thm:main_theorem 1} and \ref{Thm:main_theorem 2}.
\end{remark}

\section{Theory of the associated linear system}\label{Section:linear}

In this section, we first consider the linear system associated with 
\eqref{1st_order_sys_quasilinear}. Given $v\in H^r$,  define the
operators $\fU$   by $\fU(v) u = -\cA^i (v) \partial_i u$.
With this notation,  \eqref{1st_order_sys_quasilinear}  can be recast as
\begin{equation}
\label{1st_order_sys}
\partial_t u -\fU(u)u  = \cR(u), \quad u(0)=u_0,
\end{equation}

Our main goal of this section is to prove the following energy estimates. 
\begin{proposition}\label{Prop:energyest}
Let $\displaystyle s >  N/2+1$, $I=[0,T] \subset \bR$ for some $T>0$.
Given any  $u,v \in C^\infty(I \times \bT^N)$ satisfying  $\bigcup_{t\in I} {\rm Im}(v(t))\subset U$ and 
\begin{align}
\partial_t u -\fU(v)u   & = \cR(v),   \quad u(0)=u_0,
\label{1st_order_sys_eq}
\end{align}
the following energy estimate holds  for all $t \in I$ 
\begin{align}
\label{Basic EE}
\norm{u(t)}_s^2 \leq    F_I(\norm{v(t)}_\infty) e^{ t   F_I(\norm{v}_{\bE^s(I)}) } \left [\norm{u_0}_s^2 +\int_0^t  \left( C_I+ F_I(\norm{v (\tau)}_\infty)\norm{v (\tau)}_s^2 \right)\, d\tau   \right ].
\end{align}
Similar results holds on $[-T,0]$.
\end{proposition} 
\begin{proof}
Note that the operator $\fS =\fS(v ):=Op(\cS^* \cS) \in \cL(H^0)$ is self-adjoint, where
\begin{align}
Op(\cS^* \cS)f(x):= \sum\limits_{\xi\in \bZ^N} e^{\ii x\cdot \xi}  (\cS^* \cS)( t,  x,v( t,  x),\xi)\widehat{f}(\xi) .
\nonumber
\end{align}
Put   $N_s:=N_s(v)=\la \nabla \ra^s  \fS(v)\la \nabla \ra^s $, $\tilde{u}= \la \nabla \ra ^s  u$, and $\tilde{\cR}=\la \nabla \ra ^s  \cR$.
We compute
\begin{align}
\begin{split}
\frac{d}{dt} (N_s u, u)   
 =& 
( \fS\la \nabla \ra^s \partial_t u, \la \nabla \ra^s u)
+
(\fS \la \nabla \ra^s u , \la \nabla \ra ^s \partial_t  u)
+
(N_s' u, u )
\\
=& 
( \fS\la \nabla \ra ^s \fU u, \la \nabla \ra ^s u)
+
(\fS\la \nabla \ra ^s  \cR ,  \la \nabla \ra ^s u)\\
&+
( \fS\la \nabla \ra ^s u, \la \nabla \ra ^s \fU u  )  
 + 
(\fS\la \nabla \ra ^s u, \la \nabla \ra ^s \cR  )
+
(N_s' u, u )
\\
=& 
( \fU \tilde{u}, \fS\tilde{u})
+
( \fS\tilde{u}, \fU \tilde{u} )
+ 
(\fS\tilde{\cR} ,  \tilde{u} )
+ 
(\fS\tilde{u} , \tilde{\cR} )\\
&  +  ( \fR u,   \fS\tilde{u}) + ( \fS \tilde{u},  \fR u)
+
(N_s' u, u ),
\end{split}
\nonumber
\end{align}
where $N_s'=\frac{d}{dt}N_s =\frac{d}{dt}N_s (t,v(t)) $ and $\fR=\fR(v)= [\la \nabla \ra ^s , \fU(v)] $.
We infer from (A1) and Remark~\ref{Rmk: est of A and R} that
\begin{align}
\nonumber
\begin{split}
\|\fR(v) u \|_0 \leq     ( C_I+F_I(\norm{v}_\infty)\|v\|_s  )\|   u\|_{1,\infty} + F_I(\norm{v}_\infty)\norm{v}_{1,\infty} \|u\|_s   .
\end{split}
\end{align}
By Plancherel's Theorem, we have
\begin{align}
\begin{split}
( \fU \tilde{u}, \fS\tilde{u}) = \ii  (\cA \widehat{\tilde{u}}, \cS^* \cS \widehat{\tilde{u}}) =\ii (\cS^* \cS\cA \widehat{\tilde{u}},  \widehat{\tilde{u}}) 
= \ii(\cS^* \cD \cS  \widehat{\tilde{u}},  \widehat{\tilde{u}}) 
\end{split}
\nonumber
\end{align}
and 
\begin{align}
\begin{split}
( \fS\tilde{u}, \fU \tilde{u} ) =( \widehat{\tilde{u}},  \ii \cS^* \cD \cS  \widehat{\tilde{u}})=-\ii (\cS^* \cD \cS  \widehat{\tilde{u}},  \widehat{\tilde{u}}) =-( \fU \tilde{u}, \fS\tilde{u}) .
\end{split}
\nonumber
\end{align}
Further, by Plancherel's Theorem, H\"older's and Young's inequalities, and Lemma~\ref{Lem: P},
\begin{align}
\nonumber
(\fS\tilde{\cR}, \tilde{u})  \leq    F_I(\norm{v}_\infty) \left(  \| \cR(v)\|_s^2  +  \| u\|_s^2 \right)   
\end{align}
and
\begin{align}
\nonumber
 (\fR u, \fS \tilde{u})  
\leq &   F_I(\norm{v}_{C(I;C^1)}) \| u\|_s^2   + F_I(\norm{v}_\infty)   \norm{u}_{1,\infty} \norm{u}_s +   F_I(\norm{v}_\infty) \norm{v}_s \norm{u}_{1,\infty} \norm{u}_s \\
\leq &   F_I(\norm{v}_{C(I;C^1)}) \| u\|_s^2 + F_I(\norm{v}_\infty)  \norm{u}_{1,\infty}^2 +   F_I(\norm{v}_\infty) \norm{v}_s^2 \norm{u}_{1,\infty}^2.  
\end{align}
The terms $(\fS\tilde{u} , \tilde{\cR} ) $  and $ ( \fS \tilde{u},  \fR u)$ can be treated in an analogous way.

It follows from the Plancherel Theorem, the dominated convergence theorem and Lemma~\ref{Lem: P} that
\begin{align}
\begin{split}
(N_s' u , u )=& (\partial_t (\cS^* \cS) \widehat{\tilde{u}}, \widehat{\tilde{u}}) \leq F_I(\norm{v}_\infty)(1 +  \|\partial_t v \|_{C(I\times\bT^N)} )  \|u\|_s^2.
\end{split}
\nonumber
\end{align}
Summarizing the above computations yield
\begin{align}
\label{EE inequality}
\begin{split}
&\frac{d}{dt} (N_s u, u)\\
 \leq & 
  F_I(\norm{v}_{\bB(I)}) \norm{u}_s^2 +F_I(\norm{v}_\infty)  \norm{u}_{1,\infty}^2 +  F_I(\norm{v}_\infty) \|\cR(v)\|_s^2  
+       F_I(\norm{v}_\infty) \norm{v}_s^2 \norm{u}_{1,\infty}^2   .
\end{split}
\end{align}
Finally, it follows from \eqref{EE inequality}, the Gr\"onwall's inequality and the Sobolev embedding that
\begin{align}
\label{Basic EE Ns}
\begin{split}
  (N_s(v(t))u(t),u(t)) 
 \leq      e^{ t  F_I(\norm{v}_{\bE^s(I)}) } \left [ (N_s(v(0))u_0,u_0) + F_I(\norm{v(t)}_\infty) \int_0^t    \norm{\cR(v(\tau))}_s^2  \, d\tau   \right ] 
\end{split}
\end{align}
and further by \eqref{norm equivalence}
\begin{align}
\label{Basic EE intermediate step}
\norm{u(t)}_s^2 \leq   F_I(\norm{v(t)}_\infty) e^{ t   F_I(\norm{v}_{\bE^s(I)}) }\left [\norm{u_0}_s^2 +\int_0^t    \norm{\cR(v(\tau))}_s^2  \, d\tau   \right ] .
\end{align}
The proof is completed by applying Remark~\ref{Rmk: est of A and R} to \eqref{Basic EE intermediate step}.
\end{proof}

\medskip
\begin{remark}\label{Rmk: energy est s=0}
In the special case $s=0$,  we have $\fR=0$. Then  \eqref{EE inequality} becomes 
\begin{align}
\label{EE inequality s=0}
\begin{split}
\frac{d}{dt} (N_0 u, u)
 \leq & 
  F_I(\norm{v(t)}_{\bB(I)}) \norm{u}_0^2  + F_I(\norm{v}_\infty) \|\cR(v)\|_0^2  .
\end{split}
\end{align}
It follows from the Gr\"onwall's inequality that
\begin{align}
\label{Basic EE intermediate step s=0}
\norm{u(t)}_0^2 \leq   F_I(\norm{v}_\infty) e^{ t   F_I(\norm{v}_{\bB(I)}) }\left [\norm{u_0}_0^2 +\int_0^t    \norm{\cR(v(\tau))}_0^2  \, d\tau   \right ] .
\end{align}
This  energy inequality applies to all $u,v\in \bB(I)$ and  will be used to prove uniqueness of solutions.
\end{remark}

 \medskip
\begin{remark}
\label{RMK: condition on r}
The necessity of the condition  $s>N/2+1$ in Proposition~\ref{Prop:energyest} stems from  the estimate of the commutator $\fR(v)$.
\end{remark}

\medskip
The following corollary can be proved in a similar way to symmetric hyperbolic systems (see, e.g., \cite{Kato-1970,Kato-1973}).

\begin{corollary}\label{Cor: linear system}
Let $T>0$ and $I=[-T,T]$.
Given $v\in C^\infty(I\times\bT^N)$ with $\bigcup_{t\in I}{\rm Im}(v(t))\subset U$ and $u_0\in C^\infty(\bT^N)$, there exists a unique solution $u\in C^\infty(I\times\bT^N)$ solving \eqref{1st_order_sys_quasilinear}.
\end{corollary}


\section{Local existence and uniqueness of the quasilinear system}\label{Section:existence and uniqueness}

In this section, the notation $I$ always denotes a closed bounded interval $[-T,T]$ for some $T>0$. $t$ is assume to be in $[0,T]$ unless mentioned otherwise. All the estimates established for $t\in [0,T]$ hold in $[-T,0]$ simply by time reversal.

\subsection{Approximating sequence}\label{Section 5.1}
We take a sequence of smooth initial data $u_{0,n} \to u_0$ in $H^r$ with $r> N/2+1$ and $\tm:={\rm dist}(U^c, {\rm Im}(u_0))>0$.
Then we inductively study
\begin{align}
\label{1st_order_sys_ind}
\partial_t u_n -\fU (u_{n-1}) u_n   &= \cR(u_{n-1}), \quad 
u_n(0)=u_{0,n}.
\end{align}
By Corollary~\ref{Cor: linear system}, $u_n\in C^\infty(I\times \bT^N)$ for  $T>0$ so small that $\bigcup_{t\in I}{\rm Im}(u_{n-1}(t))\subset U$. 
In the following, we will show that the existence time $T$ is uniform for all $n$.

Let $\norm{u_0}_r  \leq K$.
Without loss of generality, we may assume
\begin{align}
\label{asp on approx initial}
\norm{u_{0,n} - u_0}_r  <  \frac{\tm}{2 C},
\end{align}
where $C$ is the embedding constant of $H^{r-1}$ into $C(\bT^N)$.
We   make the inductive assumption 
\begin{align}
\nonumber
H(k): u_k\in C^\infty(I\times \bT^N) \text{ solves \eqref{1st_order_sys_ind} and } \norm{u_k}_{C(I;H^r)} \leq \cC_1 \, \text{ and } \, 
\|\partial_t u_k\|_{C(I;H^{r-1})}\leq \cC_2 .
\end{align}
for $k=1,2,\cdots, n-1$.
Let $\cC=2\cC_1+\cC_2$. 
$H(k)$ and the Sobolev embedding theorem show that, for all $k=1,2,\cdots, n-1$
$$
\norm{u_k- u_{0,k}}_{C(I\times \bT^N)} \leq C T \|\partial_t u_k \|_{C(I;H^{r-1})} \leq C T\cC_2 ;
$$
and in view of \eqref{asp on approx initial}, by taking $T$ sufficiently small, we can make
\begin{align}
\nonumber
{\rm Im}(u_k)\subset U \quad \text{and} \quad \norm{u_k}_{C(I\times \bT^N)} \leq CK + \tm,  \quad k=1,2,\cdots, n-1.
\end{align}
As a consequence, $u_n\in C^\infty(I\times \bT^N)$ and \eqref{Basic EE} becomes
\begin{align}
\label{Energy estimate un}
\norm{u_n(t)}_r^2 \leq  F_I(CK+\tm) e^{ t    F_I(\cC) }\left [\norm{u_{0,n}}_r^2 +\int_0^t  \left( C_I+ 
F_I(CK + \tm)\norm{u_{n-1}(\tau)}_r^2 \right)\, d\tau   \right ] .
\end{align}

Furthermore, we choose $\cC_1$ in $H(k)$  in such a way that
\begin{align}
\label{ASP on C1}
\begin{split}
\sqrt{F_I(CK+\tm) (2K^2 +2\tm K/C+2\tm^2/C^2)} =  \cC_1 .
\end{split}
\end{align}
We will specify the choice of $\cC_2$ later. We would like to point out this will not cause any problem as long as we are willing make $T$ sufficiently small.

Now we will use \eqref{Energy estimate un} to estimate 
\begin{align}
\norm{u_n(t)}_r^2 \leq  &     F_I(CK+\tm)  e^{ t   F_I(\cC) }\left [(K+\tm/2C)^2 + t  \left( C_I+ F_I(CK + \tm)\cC_1^2 \right)   \right ] .
\nonumber
\end{align}
By choosing $T$ small enough, we can control
\begin{align}
\nonumber
\norm{u_n(t)}_r^2  \leq F_I(CK+\tm) (2K^2 +2\tm K/C+2\tm^2/C^2)= \cC_1^2  ,
\end{align}
which gives
$
\norm{u_n}_{C(I;H^r)} \leq \cC_1.
$
We plug this estimate into \eqref{1st_order_sys_ind} and thus obtain
\begin{align}
\label{ASP on cC2}
\begin{split}
\| \partial_t u_n(t)\|_{r-1} \leq & (C_I+ F_I(\norm{u_{n-1}(t)}_\infty) \norm{u_n(t)}_{r-1} )  \cC_1 + C_I +F_I(\norm{u_{n-1}(t)}_\infty) \norm{u_n(t)}_{r-1}  \\
\leq &  C_I \cC_1 +F_I(CK+\tm)\cC_1^2 + C_I   +F_I(CK+\tm)\cC_1 =: \cC_2.
\end{split}
\end{align}
This completes the verification of $H(n)$. One thus infers that  
\begin{align}
\norm{u_n}_{\bE^r(I)} \leq \cC   \quad \text{for all } n\in \bN.
\label{Bound EE}
\end{align}

\subsection{Convergence}\label{Subsection: convergence}
The theoretical basis of the convergence of $u_n$ is the following lemma on an energy estimate for the difference of two solutions to \eqref{1st_order_sys_eq}.
\begin{lemma}
\label{Lem: difference}
Assume that $v_i,h_i\in \bE^r(I)$ and $v_{0,i}\in H^r$ with $i=1,2$ satisfy
\begin{align}
\nonumber
\partial_t v_i -\fU (h_i) v_i   &= \cR(h_i), \quad 
v_i(0)=v_{0,i}.
\end{align}
Then their difference $w=v_1-v_2$ satisfies the estimate
\begin{align}
\begin{split}
\norm{w}_{C(I;H^0)}^2 & \leq  F_I(\norm{h_1}_{C(I\times \bT^N)}) e^{ T F_I(\norm{h_1}_{\bB(I)})} 
\Big[ \norm{v_{0,1}-v_{0,2}}_0^2   \\
&+ T F_I\left(\norm{h_1 }_{C(I;H^0)}+ \norm{h_2 }_{C(I;H^0)}+ \|v_2 \|_{C(I;H^r)}  \right)\norm{(h_1-h_2) }_{C(I;H^0)}^2   \Big].
\label{EE_diff_0}
\end{split}
\end{align}
\end{lemma}
\begin{proof}
First, note that $w$ solves
\begin{align}
\label{1st_order_sys_diff}
\partial_t w &= \fU(h_1) w  + \fF , \quad w(0) =v_{0,1}-v_{0,2}, 
\end{align}
where
\begin{align}
\nonumber
\fF:=\fF(h_1,h_2,v_2)=[\fU(h_2)-\fU(h_1)] v_2  +\cR(h_1)-\cR(h_2).
\end{align}
By \eqref{Basic EE intermediate step s=0},  we have
\begin{align}
\nonumber
\norm{w(t)}_0^2 \leq F_I(\norm{h_1(t)}_\infty) e^{ t F_I(\norm{h_1}_{\bB(I)})} 
\left[ \norm{v_{0,1}-v_{0,2}}_0^2 + \int_0^t \norm{\fF(h_1,h_2,v_2)(\tau)}_0^2\, d \tau \right].
\end{align}
Based on (A1) and (A2),  it follows from \eqref{asp on approx initial} and \eqref{Bound EE} that
\begin{align}
\label{difference est 1}
\begin{split}
 \|[\fU(h_2)-\fU(h_1)] v_2\|_0  
\leq  & F_I(\norm{h_1}_\infty+ \norm{h_2}_\infty)\norm{h_1-h_2}_0 \|v_2\|_{1,\infty} ,
\end{split}
\end{align}
and
\begin{align}
\label{difference est 3}
\|\cR(h_2)-\cR(h_2)\|_0 \leq F_I(\norm{h_1}_\infty+ \norm{h_2}_\infty)\norm{h_1-h_2}_0.
\end{align}
We infer from \eqref{difference est 1}  and \eqref{difference est 3} that
\begin{align}
\nonumber
\begin{split}
\norm{w(t)}_0^2 \leq & F_I(\norm{h_1(t)}_\infty) e^{ t F_I(\norm{h_1}_{\bB(I)})} 
\Big[ \norm{v_{0,1}-v_{0,2}}_0^2 \\
&+ \int_0^t F_I(\norm{h_1(\tau)}_\infty+ \norm{h_2(\tau)}_\infty)\norm{(h_1-h_2)(\tau)}_0^2 (\|v_2(\tau)\|_{1,\infty} +1 )^2\, d\tau \Big].
\end{split}
\end{align}
Taking supremum on both sides over $I$ yields \eqref{EE_diff_0}.
\end{proof}
Taking $w=u_n-u_{n-1}$, $v_1=u_n$, $v_2=h_1=u_{n-1}$ and $h_2=u_{n-2}$ in \eqref{EE_diff_0} yields
\begin{align}
\nonumber
\norm{u_n-u_{n-1}}_{C(I;H^0)}   
\leq    F_I(\cC) e^{ \frac{T}{2} F_I(\cC) } \left( \norm{u_{0,n}-u_{0,n-1}}_0   
  + \sqrt{ T}    \norm{u_{n-1}- u_{n-2}}_{C(I;H^0)}  \right) .
\end{align}
In the above, we have used the fact that   $F_I(\norm{u_k}_{\bB(I)})$  can be made uniformly bounded by $F_I(\cC)$.
By taking $T$ sufficiently small and choosing $\{u_{0,n}\}_{n=1}^\infty$ properly, 
we can make
$$
\norm{u_n-u_{n-1}}_{C(I;H^0)}  \leq 2^{-n} + \frac{1}{2}\norm{u_{n-1}- u_{n-2}}_{C(I;H^0)}  \leq  \frac{n-1}{2^n} + \frac{1}{2^{n-1}}\norm{u_1- u_0}_{C(I;H^0)} .
$$
One can infer from the above inequality that $\{u_n\}_{n=1}^\infty$ is Cauchy in $C(I,H^0)$. By the interpolation theory and \eqref{Bound EE}, one can further conclude that  $\{u_n\}_{n=1}^\infty$ is Cauchy in $C(I,H^{r-1})$.
Using \eqref{1st_order_sys_diff}, we get
\begin{align*}
\norm{\partial_t (u_n-u_{n-1}) (t)}_0 \leq \cC \left(   \norm{(u_n-u_{n-1})(t)}_1   +  \norm{(u_{n-1}-u_{n-2})(t)}_1 \right).
\end{align*}
This implies that $\{u_n\}_{n=1}^\infty$ is Cauchy in $\bE^1(I)$. Using the interpolation theory and \eqref{Bound EE} once more, one can conclude that $\{u_n\}_{n=1}^\infty$ is Cauchy in $\bE^{r-1}(I)$ and thus   
\begin{align}
\nonumber
u_n\to u \quad \text{in }\bE^{r-1}(I)
\end{align}
for some $u\in \bE^{r-1}(I)$. We can let 
$n\to \infty$ in \eqref{1st_order_sys_ind} and thus $u$ satisfies
\begin{align}
\partial_t u -\fU (u) u   &= \cR(u), \quad u(0)=u_0.
\nonumber
\end{align}
Next, notice that it follows from \eqref{Bound EE} that
\begin{align}
\norm{u }_{\bE^r(I)} \leq \cC .
\nonumber
\end{align}
 
 \medskip
\begin{remark}
\label{RMK: Uniform bound}
An important observation is that by our choice of $\cC_1$ and $\cC_2$ in \eqref{ASP on C1} and \eqref{ASP on cC2}, they depends only on    $\|u_0\|_r$ and
${\rm dist}(U^c, {\rm Im}(u_0))$. 
Consequently, the existence time $T$ and the bound $\cC$ is determined by $\|u_0\|_r$ and
${\rm dist}(U^c, {\rm Im}(u_0))$. 
\end{remark}

\begin{remark}
If we directly estimate $\norm{u_n-u_{n-1}}_{\bE^1(I)}$ as in the symmetric case, the following estimate similar to \eqref{EE inequality s=0}:
\begin{align}
\nonumber
\begin{split}
\frac{d}{dt} (N_1 u, u)
 \leq & 
  F_I(\norm{v(t)}_{\bB(I)}) \norm{u}_1^2  + F_I(\norm{v}_\infty) \|\fF\|_1^2  
\end{split}
\end{align}
can be used to estimate $\norm{u_n-u_{n-1}}_{C(I,H^1)}$, which yields
\begin{align}
\nonumber
\norm{u(t)}_1^2 \leq   F_I(\norm{v}_\infty) e^{ t   F_I(\norm{v}_{\bB(I)}) }\left [\norm{u_0}_1^2 +\int_0^t    \norm{\fF(\tau)}_1^2  \, d\tau   \right ] .
\end{align}
In this approach, we will have to bound $\norm{\fF(\tau)}_1$ by $\norm{v_2}_{2,\infty}$, which enforces $r>n/2+2$. In \cite{Bemfica-Disconzi-Rodriguez-Shao-2021}, \eqref{Basic EE} was used, which again requires $r>n/2+2$. Here, instead of estimating $\norm{u_n-u_{n-1}}_{\bE^1(I)}$ directly, we first bounded $\norm{u_n-u_{n-1}}_{C(I,H^0)}$ by using \eqref{EE inequality s=0} and then prove $\{u_n\}$ is Cauchy in $C(I,H^{r-1})$ by using interpolation theory. This further implies that  $\{u_n\}$ is Cauchy in $\bE^{r-1}(I)$.
\end{remark}

\subsection{Uniqueness}
Assume that $\tilde{u}\in \bB(I) $ is another solution to \eqref{1st_order_sys_quasilinear}. Then $w=u-\tilde{u}$ solves
\begin{align}
\nonumber
\partial_t w  = \fU(u) w  
 + \fF, \quad
w(0) =0  
\end{align} 
with $\fF=[\fU(u)-\fU(\tilde{u})] \tilde{u}  +\cR(u)-\cR(\tilde{u})$. 
Put
$ 
\cC_3= \norm{u}_{\bB(I)}  + \norm{\tilde{u}}_{\bB(I)} .
$ 
From Remark~\ref{Rmk: energy est s=0} and a similar argument leading to \eqref{difference est 1} and \eqref{difference est 3}, we derive that
\begin{align}
\nonumber
\norm{w(t)}_0^2 \leq F_I(\cC_3)  e^{ t F_I(\cC_3)  }\int_0^t \norm{\fF(\tau)}_0^2\, d\tau \leq  t F_I(\cC_3)  e^{ t F_I(\cC_3)    }    \norm{w }_{C([0,t];H^0)}^2  
\end{align}
for all $t\in [0,T]$. 
This implies that 
\begin{align}
\nonumber
 \norm{w }_{C([0,t];H^0)}^2    \leq  t F_I(\cC_3)  e^{ t F_I(\cC_3)    }    \norm{w }_{C([0,t];H^0)}^2 . 
\end{align}
Choosing $T_0$ sufficiently small so that $T_0 F_I(\cC_3)  e^{ T_0 F_I(\cC_3)    }  <1$ shows that  $u(t)=\tilde{u}(t)$ for $t\in [0,T_0]$.
Now we can apply the same strategy for $w(t) = u( T_0 +t ) - \tilde{u}(T_0+t)$. 
Repeating this argument for finitely many times and using time reversal reveal  that  $u(t)=\tilde{u}(t)$  on $I$.

\subsection{Continuity of solution}\label{Section: continuity}
The fact that the solution $u\in \bE^r(I)$ is an immediate consequence of the following proposition.
\begin{proposition}\label{Prop: continuity}
Suppose that $v_n,w_n\in C^\infty(I\times \bT^N)$ and $w_{0,n}\in H^r$ such that
$$
\partial_t w_n -\fU(v_n)w_n=\cR(v_n),\quad w_n(0)=w_{0,n} .
$$
Assume further that $ \bigcup\limits_{t\in I } {\rm Im}(v_n(t)) \subset U$ and
\begin{equation}
\label{continuity bdd}
\|v_n\|_{\bE^r(I)} +\|w_n\|_{\bE^r(I)} \leq C_I
\end{equation}
and as $n\to \infty$
\begin{equation}
\label{continuity converg}
w_n\to w, \,  v_n \to v \quad \text{in } \bE^{r-1}(I)  \quad\text{and}\quad  w_{0,n}\to w_0,\,  v_n(0)\to v_0 \quad\text{in } H^r.
\end{equation}
Then 
\begin{equation}
\label{continuity sys}
\partial_t w  -\fU(v )w =\cR(v ),\quad w (0)=w_0
\end{equation}
and $w\in \bE^r(I)$.
\end{proposition}
\begin{proof}
It is clear that $w$ solves \eqref{continuity sys}. It remains to show that $w\in \bE^r(I)$.
The weak continuity of $[t\to w(t)]$ can be proved by a similar argument to that of 
quasilinear wave equations, since in that proof the structure of the equation is not necessary but only the convergence $w_n\to w $ in $C(I; H^{r-1})\cap C^1(I; H^{r-2})$ and an estimate of the form~\eqref{continuity bdd} matter.
This implies  that the map $[t\to N_r(v_0)w(t)] $ is weakly continuous in $H^{-r}$.
Given any $\phi\in H^r$, 
\begin{align}
&(N_r(v(t))w(t) - N_r(v_0) w_0  , \phi) 
\\
=& (N_r(v_0)w(t) - N_r(v_0) w_0  ,  \phi) 
+ ((N_r(v(t))-N_r(v_0))w(t)   ,  \phi) .
\nonumber
\end{align}
Due to the Plancherel Theorem, Lemma~\ref{Lem: P}, \eqref{continuity bdd}, \eqref{continuity converg} and the fact that $v\in \bE^{r-1}(I)$, the second term on the right tends to $0$ and $t\to 0$. Thus  $[t\to N_r(v(t))w(t)] $ is weakly continuous in $H^{-r}$ at $t=0$.

Next,   \eqref{Basic EE Ns} implies that  
\begin{align}
\label{EE Nr un}
(N_r(v_n(t)) w_n(t), w_n(t)) \leq    e^{ t F_I(C_I)} \left [(N_r(v_n(0) ) w_{0,n}, w_{0,n}) +t F_I(C_I) \right ].
\end{align}
Observe that, for any $t\in I$, $[f\to (N_r(v(t)f,f)]$ defines an equivalent norm on $H^r$. Consequently,
\begin{align}
\label{Converg Nr un 1}
(N_r(v(t))w(t),w(t)) \leq \liminf\limits_{n\to \infty} (N_r(v(t))w_n(t),w_n(t)).
\end{align}
On the other hand,  Lemma~\ref{Lem: P} and  \eqref{continuity converg} show
\begin{align}
\label{Converg Nr un 2}
(N_r(v(t))w_n(t),w_n(t))=\lim\limits_{k\to \infty} (N_r(v_k(t))w_n(t),w_n(t)) 
\end{align}
uniformly with respect to $n$.
Taking $n\to \infty$ in  \eqref{EE Nr un} yields 
$$
(N_r(v (t)) w (t), w(t)) \leq  e^{ t F_I(C_I)} \left [(N_r(v_0 ) w_0, w_0) +t F_I(C_I) \right ]
$$
and by further pushing $t\to 0^+$
\begin{align}
\label{Bound EE 2}
\begin{split}
\limsup_{t\to 0^+}   (N_r(v (t)) w (t), w(t)) \leq 
(N_r(v_0 ) w_0, w_0).
\end{split} 
\end{align}
Direct computations show
\begin{align}
\begin{split}
&\limsup\limits_{t \to 0^+}  (N_r(v_0) (w (t) - w_0) ,   w (t) -  w_0)\\
 \leq &\limsup_{t\to 0^+}((N_r(v_0) - N_r(v (t))   )w(t),   w (t) -  w_0) +  \limsup_{t\to 0^+}(N_r(v (t)) w (t) - N_r(v_0)w_0,   w (t) -  w_0).
\end{split} 
\nonumber
\end{align}
It again follows from  the Plancherel Theorem, Lemma~\ref{Lem: P},  \eqref{continuity bdd}, \eqref{continuity converg} and    $v\in \bE^{r-1}(I)$ that the first term on the right is zero.
Meanwhile, \eqref{Bound EE 2} implies that 
\begin{align}
\nonumber
 \limsup_{t\to 0^+}(N_r(v (t)) w (t) - N_r(v_0)w_0,   w (t) -  w_0)\leq 0.
\end{align}
By \eqref{norm equivalence}, it follows that
\begin{align}
\nonumber
 \limsup\limits_{t\to 0^+} \norm{w(t)-w_0}_r^2 \leq  F_I(C_I) \limsup\limits_{t \to 0^+}  (N_r(v_0) (w (t) - w_0) ,   w (t) -  w_0)=0.
\nonumber
\end{align}
This shows the right continuity of $[t\to w(t)]$ at $t=0$ in $H^r$. Then the continuity of $[t\to w (t)]$ on $I$ in   $H^r$ follows from the time reversal and the uniqueness of solutions (which we use upon setting up the problem with data at a time different than zero).
Hence, we infer that $w   \in C(I; H^r)$.
Using this fact and equation~\eqref{continuity sys}, we immediately conclude that $w\in \bE^r(I)$.
\end{proof}

\subsection{Energy estimate}\label{Section:energy estimate}

In \eqref{1st_order_sys_ind}, applying Remark~\ref{Rmk: est of A and R} to \eqref{EE inequality} yields
\begin{align}
\nonumber
\begin{split}
&\frac{d}{dt} (N_r(u_{n-1}) u_n, u_n) \\
 \leq & 
  F_I(\norm{u_{n-1}}_{\bB(I)}) \norm{u_n}_r^2 + F_I(\norm{u_{n-1}}_{\bB(I)})  \norm{u_n}_{1,\infty}^2 +  F_I(\norm{u_{n-1}}_{\bB(I)}) \\  
&+       F_I(\norm{u_{n-1}}_\infty)  (\norm{u_n}_{1,\infty}^2 +1)\norm{u_{n-1}}_r^2   \\
\leq  &   F_I(\norm{u_{n-1}}_{\bB(I)}) \norm{u_n}_r^2  
 +     F_I(\norm{u_{n-1}}_{\bB(I)})F_I(\norm{u_n}_{\bB(I)})  \left(\norm{u_{n-1}}_r^2  +1 \right).
\end{split}
\end{align}
We have suppressed the time variable $t$ in the above estimate. This implies that
\begin{align}
\nonumber
\begin{split}
  f_n(t)   
 \leq  &  f_n(0) +  t   F_I(\norm{u_{n-1}}_{\bB(I)}) F_I(\norm{u_n}_{\bB(I)})  \\
& +  F_I(\norm{u_{n-1}}_{\bB(I)})F_I(\norm{u_n}_{\bB(I)}) \int_0^t    ( f_n(\tau)  +f_{n-1}(\tau) )\, d\tau,
\end{split}
\end{align}
where $f_n(t)=(N_r(u_{n-1}(t)) u_n(t), u_n(t))$.
Since  
$$
\lim\limits_{n\to \infty} F_I(\norm{u_n}_{\bB(I)})=F_I(\norm{u}_{\bB(I)}) ,\quad  \lim\limits_{n\to \infty} (N_r(u_{0,n-1}) u_{0,n}, u_{0,n})=(N_r(u_0) u_0, u_0), 
$$
we conclude that
\begin{align}
\nonumber
\limsup\limits_{n\to \infty}f_n(t)\leq  &  (N_r(u_0) u_0, u_0) + t    F_I(\norm{u}_{\bB(I)})  
 + F_I(\norm{u}_{\bB(I)}) \int_0^t     \limsup\limits_{n\to \infty}f_n(\tau)\, d\tau.
\end{align}
Combining with \eqref{Converg Nr un 1}, \eqref{Converg Nr un 2} and  the  Gr\"onwall's inequality, this implies that
\begin{align}
\nonumber
\begin{split}
&(N_r(u(t)) u(t), u(t)) \\
\leq & \limsup\limits_{n\to \infty}f_n(t)\leq  e^{t   F_I(\norm{u}_{\bB(I)}) } \left[ (N_r(u_0) u_0, u_0) + t     F_I(\norm{u}_{\bB(I)}) \right] .
\end{split}
\end{align}
In view of \eqref{norm equivalence}, we obtain the energy estimate~\eqref{Thm 1:energy estimate}.

To prove the estimate~\eqref{Thm 1:energy estimate 2}, we start with \eqref{EE inequality}, which implies 
\begin{align}
\nonumber
\begin{split}
  \|u(t)\|_r^2
 \leq   F_I(\norm{u}_{C(I\times \bT^N)})    \|u_0\|_r^2 + F_I(\norm{u}_{C(I\times \bT^N)}) \int_0^t F_I(\norm{u(\tau)}_r^2)     \, d\tau 
\end{split}
\end{align}
in virtue of \eqref{norm equivalence}.
Then \eqref{Thm 1:energy estimate 2} follows from the generalized Gr\"onwall inequality, cf. \cite{Bihari-1956}.


\subsection{Continuous dependence on initial data}\label{Section:continuous dependence}
 
In this subsection, we will give the proof of Theorem~\ref{Thm:main_theorem 3}. For simplicity, we will only consider the case $I=[0,T]$. The argument on the interval $[-T,0]$ follows  simply by time reversal.  
The proof   can be obtained by an analogous argument to that in \cite[Section~4.5]{Kato-1975-1}, where the author dealt with symmetric hyperbolic systems. 
The only necessary  change for the system~\eqref{1st_order_sys_quasilinear} is to show that the   family of operators $\{-\fU(u(t))\}_{t\in I}$, c.f. \eqref{1st_order_sys}, where $u(t)$ is the solution corresponding to an initial datum $u_0\in H^r$ with $\tm={\rm dist}(U^c, {\rm Im}(u_0))>0$, is quasi-stable in $H^0$ in the following sense.

\begin{definition}[Kato, \cite{Kato-1973}]
Let $E$ be a Banach space.  A family of closed operators $\{A(t)\}_{t\in I}$  in $E$  is called  {\em quasi-stable}   in $E$ with stability index $(M,\omega)$ for some constant $M>0$ and  (Lebesgue) upper-integrable function $\omega:I\to \bR$ if
$$
(\omega(t),\infty) \subset \rho(A(t)), \quad t\in I
$$
and
$$
\|  (\alpha_j-\omega(t_j))^k \prod_{j=1}^k  (\alpha_j + A(t_j))^{-1}  \|_{\cL(E)}  \leq M, 
$$
	for every finite family of real numbers $\{t_j,\alpha_j \}_{j=1}^k$ such that 
$0\leq t_1 \leq t_2 \leq \cdots \leq t_k\leq T$ and $\alpha_j > \omega(t_j)$, where $\rho(A(t))$ is the resolvent set of $A(t)$. 
The above product is time-ordered, i.e. a factor with larger $t_j$ stands to the left of the ones with smaller $t_j$.
\end{definition} 

In the rest of this subsection, we assume that the constant $M$   depends only on $\norm{u_0}_r $ and $\widetilde{M}$.
\begin{lemma}
The  family $\{-\fU(u(t))\}_{t\in I}$ is quasi-stable in $H^0$.
\end{lemma}
\begin{proof}
For notational brevity, we set 
$$
\cN(u)= \cN(x, u,\xi) ,\quad \cN \in \{\cA,\cS, \cD\},
$$
see (A3).
If  $  (\lambda-\fU(u(t)))  v=f\in H^0$, 
(A3) implies that 
\begin{align*}
(\lambda -i \cD(u(t)))  \cS(u(t))\widehat{v}= \cS(u(t) )\widehat{f}.
\end{align*}
one can easily show that  $\{A(u(t)\}_{t\in I}$ satisfies
$$
\|     (\lambda -\fU(u(t)))^{-1}  \|_{\cL(X_t)}  \leq  \lambda^{-1}  , \quad \lambda>0,   
$$
where $X_t=H^0$ equipped with the norm
$$
\| v \|_{X_t} :=   \| \cS(u(t))  \widehat{v} \|_0.
$$
Given any $v\in H^0$, since
$
\frac{d}{dt} ( N_0(u(t)) v , v)= ( N_0'(u(t)) v , v),
$
where the operator $N_0$ is defined in the proof of Proposition~\ref{Prop:energyest}, we infer that
\begin{align*}
\frac{d}{dt}  ( N_0(u(t)) v , v) \leq M  ( N_0(u(t)) v , v) .
\end{align*}
This implies that
\begin{align*}
\norm{v}_{X_t} \leq \norm{v}_{X_s}  e^{M(t-s)}, \quad 0\leq s \leq t \leq T.
\end{align*}
Given any $f\in H^0$, $0\leq t_1 \leq t_2 \leq \cdots \leq t_k\leq T$ and $\lambda_j>0$, we have
\begin{align*}
\|  \prod_{j=1}^k  (\lambda_j -\fU(u(t_j)))^{-1} f \|_{X_T}
\leq  & e^{M(T-t_k)} \|  \prod_{j=1}^k  (\lambda_j -\fU(u(t_j)))^{-1} f \|_{X_{t_k}}\\
\leq & e^{M(T-t_k)} \lambda_k^{-1} \|  \prod_{j=1}^{k-1}  (\lambda_j -\fU(u(t_j)))^{-1} f \|_{X_{t_k}}\\
  & \vdots \\
\leq &  e^{M(T-t_k)}  e^{M(t_k-t_{k-1})}  \cdots  e^{M t_1}  \prod_{j=1}^k\lambda_j^{-1}\norm{f}_0\\
\leq &  \prod_{j=1}^k\lambda_j^{-1} e^{MT} \norm{f}_0.
\end{align*}
In view of \eqref{norm equivalence}, this shows that $\{-\fU(u(t))\}_{t\in I}$ is quasi-stable in $H^0$ with stability index
$(M, 0)$ for some $M>0$.
\end{proof}
Let $u_{0,n}\to u_0$ in $H^r$ as $n\to \infty$ and $u_n $ be the solution of \eqref{1st_order_sys_quasilinear} with initial data $u_{0,n}$ asserted by Theorem~\ref{Thm:main_theorem 1}. Put
$$
\cB(t)=\fU(u(t))- \la \nabla \ra^r \fU(u(t)) \la \nabla \ra^{-r} , \quad \cB_n(t)=\fU(u_n(t))- \la \nabla \ra^r \fU(u_n(t)) \la \nabla \ra^{-r}.
$$
Observe that for sufficiently large $n$, $\norm{\cB(t)}_{\cL(H^0)} \leq M$ and  $\norm{\cB_n(t)}_{\cL(H^0)} \leq M$. Moreover,
$$
\norm{\cR(u(t)) - \cR(u_n(t))}_r \leq M \norm{u(t)-u_n(t)}_r ,\quad \norm{\cB(t) - \cB_n(t)}_{\cL(H^0)} \leq M \norm{u(t)-u_n(t)}_r.
$$
Now, the rest of the proof follows the argument in \cite[Section 4.5]{Kato-1975-1} line by line.

\subsection{Continuation Criterion}\label{Section:continuation criterion}

In this subsection, we will prove Theorem~\ref{Thm:main_theorem 2}. 
Given   $I=[0,T]$, if 
$u\in \bB(I)$ satisfying \eqref{Thm 1:constraint}   is a solution to \eqref{1st_order_sys_quasilinear} with $u_0\in H^r$.
Let $\Omega\subset I$ be the set of $T_0$ such that $u$ satisfies \eqref{Thm 1:regularity} and \eqref{Thm 1:energy estimate} on $[0,T_0]$.
First,   Theorem~\ref{Thm:main_theorem 1} shows that $\Omega$ contains an open neighborhood of $0$ in $I$ and thus  is nonempty. We will show that $\Omega$ is both open and closed in $I$.
To show the openness, pick  an arbitrary $T_0\in \Omega\setminus \{0\}$. By Theorem~\ref{Thm:main_theorem 1}, there exists some $\varepsilon>0$ such that \eqref{1st_order_sys_quasilinear} has a solution $u$ on $[T_0,T_0+\varepsilon]$  satisfying \eqref{Thm 1:regularity} and \eqref{Thm 1:energy estimate}.
Put
$$
I_0=[0,T_0 ], \quad I_\varepsilon=[T_0,T_0+\varepsilon]. 
$$
It is obvious that $u$ solves \eqref{1st_order_sys_quasilinear} and  satisfies \eqref{Thm 1:regularity} on $[T_0,T_0+\varepsilon]$.
Further, using \eqref{Thm 1:energy estimate} on $[0,T_0+\varepsilon]$ and  $[0,T_0]$ yields
\begin{align}
\nonumber
\begin{split}
 \norm{u(t)}_r^2 
 \leq & F_{I_\varepsilon}(\norm{u}_{\bB(I_\varepsilon)}) e^{(t -T_0)  F_{I_\varepsilon}(\norm{u}_{\bB(I_\varepsilon)}) } \left( \norm{u(T_0)}_r^2 + (t-T_0)    F_{I_\varepsilon}(\norm{u}_{\bB(I_\varepsilon)}) \right) \\
\leq & F_I(\norm{u}_{\bB(I )}) e^{(t -T_0) F_I(\norm{u}_{\bB(I )})  } \Big[ F_I(\norm{u}_{\bB(I )}) e^{T_0  F_I(\norm{u}_{\bB(I )} )} \left( \norm{u_0}_r^2  + T_0    F_I(\norm{u}_{\bB(I )} ) \right) \\
& + (t-T_0)     F_I(\norm{u}_{\bB(I )})  \Big]\\
\leq & F_I(\norm{u}_{\bB(I )}) e^{t  F_I(\norm{u}_{\bB(I )})  }\left( \norm{u_0}_r^2 + t     F_I(\norm{u}_{\bB(I )})  \right)
\end{split}
\end{align} 
for any $t\in I_\varepsilon$. Therefore, $u$ satisfies \eqref{Thm 1:energy estimate} on $[0,T_0+\varepsilon]$ and thus $\Omega$ is open in $I$.
To show the closedness of $\Omega$, assume that $T_0\in \overline{\Omega}$. 
Pick $T_0>T_n\in \Omega$ such that $T_n\to T_0$.
Consequently, both \eqref{Thm 1:regularity} and \eqref{Thm 1:energy estimate} hold on $I_n=[0,T_n]$ for all $n$. 
Since $\norm{u}_{\bB(I)}<\infty$, we infer from \eqref{Thm 1:energy estimate} and the structure of \eqref{1st_order_sys_quasilinear} that 
$$
\sup\limits_n \norm{u}_{\bE^r(I_n)} <\infty.
$$
By Theorem~\ref{Thm:main_theorem 1}, the existence time depends only on the $H^r$-norm of the initial data and
 the distance between $\partial U$ and the image of the initial data. 
We can thus extend the solution $u$ beyond $T_0$. Therefore, $T_0\in \Omega$.

Now let us prove the second part of Theorem~\ref{Thm:main_theorem 2}.
When $T^*\neq +\infty$, assume that  $M=\inf\limits_{0\leq t < T^*} {\rm dist}(\partial U , {\rm Im}u(t))>0$ and
$$
\lim\limits_{T\to T^*_-}   \norm{u}_{\bB([0,T])}<\infty.
$$
Pick arbitrary positive increasing sequence $T_n \to T^*_-$ and set $I_n=[0,T_n]$. Then, for any $n$,
$$
{\rm dist}(\partial U , {\rm Im}u(T_n))\geq M>0.
$$
Moreover, we infer from  \eqref{Thm 1:energy estimate} that 
$$
 \norm{u }_{C(I_n;H^r)} < C<\infty 
$$
for some $C$ independent of $n$.
By Theorem~\ref{Thm:main_theorem 1}, we can extend the solution $u$ beyond $T^*$. A contradiction.


\section{Applications}\label{Section:applications}

The theorems developed in previous sections are applicable to a variety of hyperbolic systems, including symmetrizable ones.
To illustrate how Theorems~\ref{Thm:main_theorem 1} and \ref{Thm:main_theorem 2} can be applied to systems of interest, we will consider three examples from relativistic fluid dynamics. We focus on relativistic fluids because seeking to cast the equations of motion in strongly hyperbolic form has been a common theme in relativity, see, e.g., \cite{Baumgarte:2010ndz,Rezzolla-Zanotti-Book-2013}, and also because it was through diagonalization that some of the results below had been first proven. Since our goal is to simply illustrate the usage of Theorems~\ref{Thm:main_theorem 1} and \ref{Thm:main_theorem 2} rather then presenting new results\footnote{Strictly speaking the regularity that can obtained for the equations of Section \ref{BDNK} is new, in that the previous result, based on \cite{Bemfica-Disconzi-Rodriguez-Shao-2021} (see also \cite{Bemfica-Disconzi-Graber-2021}) considered a version of Theorem \ref{Thm:main_theorem 1} with $r > N/2 + 2$, but this is a minor point that does not warrant a full technical statement of the results.} about relativistic fluids (the results below can be found in the cited literature), for the sake of brevity we will not state detailed theorems nor spend much time introducing notation and background, following instead the notation and convention of the references \cite{Anile-Book-1990, Gavassino:2023xkt,Bemfica:2020zjp}. In what follows, we assume familiarity with relativity theory, including its standard conventions.

\subsection{The relativistic Euler equations}\label{S:Relativistic_Euler} Following \cite{Anile-Book-1990}, we consider the relativistic Euler equations
\begin{align}
A^\alpha \partial_\alpha \Phi &= 0,
\nonumber
\end{align}
where $\Phi$ is the six-component vector $\Phi = (u^\lambda, \varrho, s)$ and the matrices $A^\alpha$ are given by
\begin{align}
\nonumber
A^\alpha = 
\begin{bmatrix}
\left.(p+\varrho) u^\alpha \updelta^\beta_\lambda \right._{
\textcolor{gray}{\, 4\times 4}} 
& 
\left.\proj^{\beta \alpha}\frac{\partial p}{\partial \varrho}\right._{\textcolor{gray}{\, 4\times 1}}
& 
\left.\proj^{\beta\alpha} \frac{\partial p}{\partial s}\right._{\textcolor{gray}{\, 4\times 1}}
\\
\left.(p+\varrho) \updelta^\alpha_\lambda\right._{\textcolor{gray}{\, 1\times 4}}
& 
\left. u^\alpha \right._{\textcolor{gray}{1\times 1}}
& 
\left. 0 \right._{\textcolor{gray}{1\times 1}}
\\
\left. 0 \right._{\textcolor{gray}{1\times 4}}
& 
\left. 0 \right._{\textcolor{gray}{1\times 1}}
& 
\left. u^\alpha \right._{\textcolor{gray}{1\times 1}}
\end{bmatrix},
\end{align}
where we indicated with subscripts and different color the size of each submatrix (observe that the index $\alpha$ labels the matrices and not their entries). Here, $u$, $\varrho$, and $s$, correspond to the (four-)velocity, (energy) density, and entropy of the fluid, respectively, and $\proj$ is the projection onto the space orthogonal (with respect to the Minkowski metric, assumed for the spacetime) to $u$; $p = p(\varrho,s)$ is the fluid's pressure.

In \cite{Anile-Book-1990}, it is proven that:
\begin{align}
\label{Diagonalization_statement}
\tag{Diag}
\parbox{5in}{%
For any timelike vector $\xi$, $\det(A^\alpha \xi_\alpha) \neq 0$, and for any spacelike vector $\chi$, the eigenvalue problem $A^\alpha( \chi_\alpha + \lambda \xi_\alpha) V = 0$ has only real eigenvalues $\lambda$ and a complete set of eigenvectors $V$,}
\end{align}
provided that 
$0 < \frac{\partial p}{\partial \varrho} \leq 1$ and $p+\varrho > 0$, which are physically natural assumptions (recall that $\sqrt{\frac{\partial p}{\partial \varrho}}$ corresponds to the fluid's sound speed). Taking $\xi = (1,0,0,0)$ and $\chi$ tangent to $\{t = \text{ constant} \}$, it follows that the system can be put in the form \eqref{1st_order_sys_quasilinear} and satisfies the assumptions of Theorem 
\ref{Thm:main_theorem 1}. The diagonalization \eqref{Diagonalization_statement} and the applicability of Theorem \ref{Thm:main_theorem 1} can be shown to hold if the metric is not Minkowski and also with coupling to Einstein's equations.

\subsection{Out-of-equilibrium relativistic bulk dynamics} This theory was recently introduced in \cite{Gavassino:2023xkt} as a promising way of modeling effects of bulk viscosity in mergers of neutron stars. For this, one needs to consider the fluid equations coupled to Einstein's equations, but here we focus on the fluid part only. The equations of motion are
\begin{align}
\partial_\mu T^\mu_\nu & = 0,
\nonumber
\\
\partial_\mu (n u^\mu) & = 0,
\nonumber
\\
\uptau u^\mu \partial_\mu \Pi + \Pi + \upzeta \partial_\mu u^\mu & = 0,
\nonumber
\end{align}
where $T_{\mu\nu}$ is the energy-momentum tensor
\begin{align}
T_{\mu\nu} = \varrho u_\mu u_\nu + (p+\Pi) \proj_{\mu\nu},
\nonumber 
\end{align}
where $u$, $\varrho$, $p$, and $\proj$ are as in Section \ref{S:Relativistic_Euler}, except that now $p = p(\varrho,n)$, $n$ is the fluid's baryon density, $\Pi$ is the fluid's bulk viscosity, $\uptau = \uptau(\varrho, n, \Pi)$ and $\upzeta = \upzeta(\varrho, n, \Pi)$ are the relaxation-time coefficient and coefficient of bulk viscosity, respectively. Using that for a relativistic fluid the velocity is normalized, $u_\mu u^\mu = -1$, decomposing $\partial_\mu T^\mu_\nu = 0$ into the directions parallel and orthogonal to $u$, and writing the resulting system in the form
$
A^\alpha \partial_\alpha \Phi = R(\Phi),
$
where $\Phi$ is the seven-component vector $\Phi = (u^\lambda, \varrho, n, \Pi)$, we find that \eqref{Diagonalization_statement} holds for this system upon taking $\xi = (1,0,0,0)$ and $\chi$ tangent to $\{t = \text{ constant} \}$, 
and thus Theorem \ref{Thm:main_theorem 1} applies, provided that $\varrho + p + \Pi , n,  \frac{\partial p}{\partial \varrho} , \frac{\partial p}{\partial n}, \uptau, \upzeta > 0 $, 
$ \frac{\partial p}{\partial \varrho} + \frac{n}{\varrho + p + \Pi} \frac{\partial p}{\partial n} \geq 0$, and  
\begin{align}
\nonumber
\frac{1}{\varrho+p+\Pi}
\left( n \frac{\partial p}{\partial n} + \frac{\upzeta}{\uptau} \right) \leq 1 - 
\frac{\partial p}{\partial \varrho}.
\end{align}
The diagonalization \eqref{Diagonalization_statement} and the applicability of Theorem \ref{Thm:main_theorem 1} can be shown to hold if the metric is not Minkowski and also with coupling to Einstein's equations.

\subsection{First-order causal relativistic viscous fluids\label{BDNK}}
This theory is the result of the works
\cite{Bemfica:2020zjp,Bemfica-Disconzi-Noronha-2019-1,Bemfica-Disconzi-Noronha-2018,Kovtun:2019hdm,Hoult:2020eho} and introduces a model of relativistic viscous fluids that contains all viscous contributions from shear and bulk viscosity and heat conduction. The equations of motion are
\begin{align}
\partial_\mu T^\mu_\nu & = 0,
\nonumber
\\
\partial_\mu (n u^\mu) & = 0,
\nonumber
\end{align}
where $T_{\mu\nu}$ is the energy-momentum tensor
\begin{align}
\nonumber
T_{\alpha\beta} := (\varrho + \mathscr{R}) u_\alpha u_\beta
+ (p + \mathscr{P}) \proj_{\alpha\beta} + \pi_{\alpha\beta} + \mathscr{Q}_\alpha u_\beta + \mathscr{Q}_\beta u_\alpha,
\end{align}
with 
\begin{align}
\mathscr{R} & := \uptau_\mathscr{R} (u^\mu \partial_\mu \varrho + (p + \varrho) \partial_\mu u^\mu ),
\nonumber
\\
\mathscr{P} & := - \upzeta \partial_\alpha u^\alpha + \uptau_\mathscr{P} (u^\mu \partial_\mu \varrho + (p+\varrho) \partial_\mu u^\mu ),
\nonumber
\\
\pi_{\alpha\beta} & := - \upeta \proj_\alpha^\mu \proj_\beta^\nu ( \partial_\mu u_\nu + \partial_\nu u_\mu - \frac{2}{3} \partial_\lambda u^\lambda g_{\mu\nu}),
\nonumber
\\
\mathscr{Q}_\alpha & := \uptau_\mathscr{Q} (p+\varrho) u^\mu \partial_\mu u_\alpha + \upbeta_\mathscr{Q} \proj_\alpha^\mu \partial_\mu \varrho + \upbeta_\varrho \proj_\alpha^\mu \partial_\mu \varrho 
+ \upbeta_n \proj_\alpha^\mu \partial_\mu n,
\nonumber
\\
\mathscr{J}_\alpha & := 0,
\nonumber
\end{align}
where  $u$, $\varrho$, $p$, and $\proj$ are as in Section \ref{S:Relativistic_Euler}, except that now $p = p(\varrho,n)$, $n$ is the fluid's baryon density, $g$ is the spacetime metric, which we are taking as Minkowski for simplicity,
\begin{align}
\upbeta_\varrho & := \uptau_\mathscr{Q} \left. \frac{\partial p}{\partial \varrho}\right|_n 
+ \upkappa \uptheta h \left. \frac{\partial (\upmu/\uptheta)}{\partial \varrho} \right|_n,
\nonumber
\\
\upbeta_n & := \uptau_\mathscr{Q} \left. \frac{\partial p}{\partial n}\right|_\varrho 
+ \upkappa \uptheta h \left. \frac{\partial (\upmu/\uptheta)}{\partial n} \right|_\varrho,
\nonumber
\end{align}
where $\uptheta = \uptheta(\varrho,n)$ is the fluid's temperature, $\upmu$ is the chemical potential determined by the thermodynamic relation
\begin{align}
\frac{dp}{p+\varrho} = \frac{d\uptheta}{\uptheta} + \frac{n \uptheta}{p+\varrho}d \left(\frac{\upmu}{\uptheta}\right),
\nonumber
\end{align}
$\uptau_\mathscr{R},\uptau_\mathscr{P},\uptau_\mathscr{Q},\upzeta, \upeta, \upkappa$ are transport coefficients that are known functions of $\varrho$ and $n$, with $\uptau_\mathscr{R},\uptau_\mathscr{P},\uptau_\mathscr{Q}$ called relaxation times and $\upzeta, \upeta, \upkappa$ being the coefficients of bulk and shear viscosity and heat conductivity, respectively. Finally, one continues to assume that the velocity satisfies the constraint $u_\mu u^\mu = -1$, which we use to decompose $\partial_\mu T^\mu_\nu = 0$ in the directions parallel and perpendicular to $u$.

Observe that the equations of motion $\partial_\mu T^\mu_\nu = 0$ are second-order. In order to apply Theorem \ref{Thm:main_theorem 1}, we have to rewrite them as a system of first-order equations. Observe also that $\partial_\mu (n u^\mu ) = 0$ is a constraint, since $n$, $u$, and their first-order derivatives are prescribed as data. This constraint is propagated by taking $u^\mu \partial_\mu$ of $\partial_\mu (n u^\mu ) = 0$ and considering the resulting equation as part of the system. 

Upon writing the system as a first-order system of the form 
$
A^\alpha \partial_\alpha \Phi = R(\Phi),
$
for the quantities 
$u^\mu \partial_\mu \varrho$, $\proj^{\alpha \mu} \partial_\mu \varrho$, 
$u^\mu \partial_\mu n$, $\proj^{\alpha \mu} \partial_\mu n$, 
$u^\mu \partial_\mu u_\alpha$, and $\proj^{\beta \mu} \partial_\mu u_\alpha$,  we find that \eqref{Diagonalization_statement} holds for this system upon taking $\xi = (1,0,0,0)$ and $\chi$ tangent to $\{t = \text{ constant} \}$, 
and thus Theorem \ref{Thm:main_theorem 1} applies, provided that the transport coefficients and the scalars in the problem satisfy a series of inequalities. These inequalities are like the ones in the above two Sections but will not be presented here for simplicity, as the complete list of inequalities is long and cumbersome (see \cite{Bemfica:2020zjp}). Solutions to the original, second-order system of equations, are obtained by a standard approximation by analytic functions and use of the Cauchy-Kovalevskaya theorem; see 
\cite{Bemfica:2020zjp} for details. The diagonalization \eqref{Diagonalization_statement} and the applicability of Theorem \ref{Thm:main_theorem 1} can be shown to hold if the metric is not Minkowski and also with coupling to Einstein's equations.

\bibliographystyle{plain}
\bibliography{References.bib}

\end{document}